\documentclass[11pt,reqno,tbtags]{amsart}
\usepackage{amssymb}
\usepackage[numeric,initials,nobysame]{amsrefs}
\usepackage{amsmath,amssymb,amsfonts,dsfont}
\usepackage{enumerate,graphicx}
\usepackage{algorithm,paralist}

\numberwithin{equation}{section}

\newtheorem{maintheorem}{Theorem}

\newtheorem{theorem}{Theorem}[section]
\newtheorem*{theorem*}{Theorem}
\newtheorem{lemma}[theorem]{Lemma}

\newtheorem{corollary}[theorem]{Corollary}

\newtheorem{definition}[theorem]{Definition}

\theoremstyle{definition}{

}

\theoremstyle{remark}{

\newtheorem*{remark*}{Remark}

}

\newcommand{\R}{\mathbb R}
\newcommand{\N}{\mathbb N}
\newcommand{\Z}{\mathbb Z}

\newcommand{\E}{\mathbb{E}}
\renewcommand{\P}{\mathbb{P}}
\DeclareMathOperator{\var}{Var}  

\renewcommand{\epsilon}{\varepsilon}

\newcommand{\cG}{{\mathcal{G}}}
\newcommand{\cB}{{\mathcal{B}}}
\newcommand{\GPC}{\cG_{\textsc{pc}}}
\newcommand{\cF}{{\mathcal{F}}}
\newcommand{\given}{\, \big| \,}

\newcommand{\one}{\boldsymbol{1}}

\newcommand{\deq}{\stackrel{\scriptscriptstyle\triangle}{=}}
\newcommand{\K}{\mathcal{K}}
\newcommand{\GC}{{\mathcal{C}_1}} 
\newcommand{\tGC}{{\tilde{\mathcal{C}}_1}}
\newcommand{\TC}[1][\mathcal{C}_1]{#1^{(2)}} 
\newcommand{\hGC}{{\widehat{\mathcal{C}}_1}}
\newcommand{\Po}{\operatorname{Po}}
\newcommand{\Bin}{\operatorname{Bin}}
\newcommand{\Geom}{\operatorname{Geom}}

\newcommand{\taupc}{\tau_{\textsc{pc}}}

\begin{document}

\title[Anatomy of the giant component]{Anatomy of the giant component: \\ The strictly supercritical regime}

\author{Jian Ding, \thinspace Eyal Lubetzky and Yuval Peres}

\address{Jian Ding\hfill\break
Department of Statistics\\
UC Berkeley\\
Berkeley, CA 94720, USA.}
\email{jding@stat.berkeley.edu}
\urladdr{}

\address{Eyal Lubetzky\hfill\break
Microsoft Research\\
One Microsoft Way\\
Redmond, WA 98052-6399, USA.}
\email{eyal@microsoft.com}
\urladdr{}

\address{Yuval Peres\hfill\break
Microsoft Research\\
One Microsoft Way\\
Redmond, WA 98052-6399, USA.}
\email{peres@microsoft.com}
\urladdr{}

\begin{abstract}
In a recent work of the authors and Kim, we derived a complete description of the largest component of the Erd\H{o}s-R\'enyi random graph $\cG(n,p)$ as it emerges from the critical window, i.e.\ for $p = (1+\epsilon)/n$ where $\epsilon^3 n\to\infty$ and $\epsilon=o(1)$, in terms of a tractable contiguous model. Here we provide the analogous description for the supercritical giant component, i.e.\ the largest component of $\cG(n,p)$ for $p = \lambda/n$ where $\lambda>1$ is fixed. The contiguous model is roughly as follows: Take a random degree sequence and sample a random multigraph with these degrees to arrive at the kernel; Replace the edges by paths whose lengths are i.i.d.\ geometric variables to arrive at the 2-core; Attach i.i.d.\ Poisson Galton-Watson trees to the vertices for the final giant component. As in the case of the emerging giant, we obtain this result via a sequence of contiguity arguments at the heart of which are Kim's Poisson-cloning method and the Pittel-Wormald local limit theorems.
\end{abstract}

\maketitle

\section{Introduction}

The famous phase transition of the Erd\H{o}s and R\'{e}nyi random graph, introduced in 1959~\cite{ER59}, addresses the \emph{double jump} in the size of the largest component $\GC$ in $\cG(n,p)$ for $p=\lambda/n$ with $\lambda>0$ fixed. When $\lambda<1$ it is logarithmic in size with high probability (w.h.p.), when $\lambda=1$ its size has order $n^{2/3}$ and when $\lambda>1$ it is linear w.h.p.\ and thus referred to as the \emph{giant component}. Of the above facts, the critical behavior was fully established only much later by Bolloba\'as~\cite{Bollobas84} and {\L}uczak~\cite{Luczak90}, and in fact extends throughout the \emph{critical window} of $p=(1\pm\epsilon)/n$ for $\epsilon = O(n^{-1/3})$ as discovered in~\cite{Bollobas84}.

As far as the structure of $\GC$ is concerned, when $p=\lambda/n$ for fixed $\lambda<1$ in fact this component is w.h.p\ a tree of a known (logarithmic) size. The structure and size of the largest components was established in~\cites{LPW} and~\cite{Aldous2}, where in the latter work Aldous showed a remarkable connection between the critical random graph, continuum random trees and Brownian excursions. (See also the recent work~\cite{ABG} further studying the component structure at criticality, as well as~\cites{Bollobas2,JLR} for further details.)

As opposed to the tree-like geometry at and below criticality, the structure of the largest component becomes quite rich as soon as it emerges from the critical window, i.e.\ at $p=(1+\epsilon)/n$ where $\epsilon=o(1)$ and $\epsilon^3 n\to\infty$.
Despite many works devoted to the study of various properties of the largest component in this regime, the understanding of its structure remained fairly limited, illustrated by the fact that one of its most basic properties --- the diameter --- was determined asymptotically only lately in~\cite{DKLP2} and independently in~\cite{RW}. In the context of our present work, out of the various decomposition results on the structure of $\GC$ it is important to mention those by {\L}uczak~\cite{Luczak91}, highlighting the kernel as a random graph with a given degree sequence, and by Pittel and Wormald~\cite{PW}, featuring very precise estimates on the distribution of the size of $\GC$ and its 2-core (The 2-core of a graph is its maximum subgraph where all degrees are at least $2$. The kernel of $\GC$ is obtained from its 2-core by replacing every maximal path where all internal vertices have degree 2 by an edge.).

Recently, the authors and Kim~\cite{DKLP1} established a complete characterization of the structure of $\GC$ throughout the emerging supercritical regime, i.e.\ when $p=(1+\epsilon)/n$ with $\epsilon^3 n \to\infty$ and $\epsilon=o(1)$. This was achieved by offering a tractable contiguous \emph{model} $\tGC$, in other words, every graph property $\mathcal{A}_n$ that is satisfied by $\tGC$ w.h.p.\ (that is, a sequence of simple graphs such that $\P(\tGC \in \mathcal{A}_n)\to 1$) is also satisfied by $\GC$ w.h.p.
The contiguous model has a particularly simple description in the early stage of the formation of the giant, namely when $p=(1+\epsilon)/n$ with $\epsilon^3 n \to\infty$ and $\epsilon=o(n^{-1/4})$:
\begin{compactenum}[(i)]
\item Sample a random 3-regular multigraph on $2\lfloor Z\rfloor $ vertices via the configuration model, where $Z$ is Guassian with parameters $\mathcal{N}(\frac23\epsilon^3 n,\epsilon^3 n)$.
\item Subdivide each edge into a path of length i.i.d.\ Geometric($\epsilon$).
 \item Attach i.i.d.\ $\mathrm{Poisson}(1-\epsilon)$-Galton-Watson trees to each of the vertices.
\end{compactenum}
(In the above, a $\mathrm{Poisson}(\mu)$-Galton-Watson tree is the family tree of a Galton-Watson branching process with offspring distribution $\mathrm{Poisson}(\mu)$. See \S\ref{sec-prelim-conf} for the definition of the configuration model.)

The advantages of the aforementioned characterization were demonstrated in two companion papers~\cites{DKLP2,DLP}. The first of these
settled the natural question of the asymptotic behavior of the diameter throughout the emerging supercritical regime\footnote{Prior to our work, \cite{RW} had independently and using a different method obtained the asymptotic diameter in most but not all of the emerging supercritical regime. Following our work they managed to close this gap. Note that the estimate there is quite precise, whereas our work only aimed to obtain the leading order term throughout the regime.}, achieved by combining the structure result with a straightforward analysis of first-passage-percolation. The second established the order of the mixing time of the random walk on $\GC$, previously known only within the critical window and in the strictly supercritical regime, lacking the interpolating regime between them. See~\cite{DKLP1} for other applications of this result to easily read off key properties of $\GC$.

In this work we provide the analogous description for the strictly supercritical giant component, i.e.\ $p=(1+\epsilon)/n$ where $\epsilon>0$ is fixed.

\begin{maintheorem}\label{mainthm-struct}
Let $\GC$ be the largest component of $\cG(n,p)$ for $p = \lambda/n$ where $\lambda>1$ is fixed.
 Let $\mu<1$ be the conjugate of $\lambda$, that is
$\mu\mathrm{e}^{-\mu} = \lambda \mathrm{e}^{-\lambda}$. Then $\GC$ is contiguous to the following model $\tGC$:
\begin{enumerate}[\indent 1.]
  \item\label{item-struct-base} Let $\Lambda$ be Gaussian $\mathcal{N}\left(\lambda - \mu, 1/n\right)$ and let $D_u \sim \mathrm{Poisson}(\Lambda)$ for $u \in [n]$ be i.i.d., conditioned that $\sum D_u \one_{D_u\geq 3}$ is even.
  Let
  \[\mbox{$N_k = \#\{u : D_u = k\}$ \quad and \quad $N= \sum_{k\geq 3}N_k$}\,.\]
  Select a random multigraph $\K$ on $N$ vertices, uniformly among all multigraphs with $N_k$ vertices of degree $k$ for $k\geq 3$.
  \item\label{item-struct-edges} Replace the edges of $\K$ by paths of i.i.d.\ $\Geom(1-\mu)$ lengths. 
  \item\label{item-struct-bushes} Attach an independent $\mathrm{Poisson}(\mu)$-Galton-Watson tree to each vertex.
\end{enumerate}
That is, $\P(\tGC \in \mathcal{A}) \to 0$ implies $\P(\GC \in \mathcal{A}) \to 0$
for any set of graphs $\mathcal{A}$.
\end{maintheorem}
(In the above, the notation $\Geom(1-\mu)$ denotes the geometric variable assuming the value $k\geq 1$ with probability $\mu^{k-1}(1-\mu)$.) We note that conditioning that $\sum D_u\one_{D_u \geq 3}$ is even can easily be realized by rejection sampling. Alternatively, this requirement can be replaced by adding a self-loop (counting 1 to the degree) to one of the vertices whenever the sum is odd.
Further note that in the above recipe for $\tGC$, Step~\ref{item-struct-base} constructs the kernel, Step~\ref{item-struct-edges} constructs the 2-core and finally the entire giant is constructed in the Step~\ref{item-struct-bushes}.

To demonstrate how one can easily derive nontrivial properties of $\GC$ from the above theorem, observe for instance that one can immediately infer that the longest path of degree-2 vertices in the 2-core is of size $\log_{1/\mu} n + O_\mathrm{P}(1)$.
Indeed, there are order $n$ edges in the kernel, hence the above quantity is simply the maximum of order $n$ i.i.d.\ geometric variables with mean $1-\mu$.

As another example, we note that it was well-known prior to this work that the giant component consists of an expander ``decorated'' using paths and trees of at most logarithmic size (see~\cite{BKW} for a concrete example of such a statement, used there to obtain the order of the mixing time on the fully supercritical $\GC$). This is immediately apparent from the above description of $\tGC$: indeed, it straightforward to show that the kernel is typically an expander (see, e.g., \cite{DKLP1}*{Lemma 3.5} where this was shown for the kernel in the emerging supercritical regime). The decorations spoiling its expansion as described in the decomposition results \`a la~\cite{BKW} are due to \begin{inparaenum}
  [(i)] \item edges subdivided into arbitrarily large paths via the geometric variables
  \item attached Poisson Galton-Watson trees of arbitrarily large size.
\end{inparaenum}
In both cases, the size of the decoration is constant in expectation (depending on $\lambda$) and has an exponential tail, reproducing the above depicted picture.

\subsection{Main techniques and comparison with~\cite{DKLP1}}
Our framework for obtaining the description of the largest component in $\cG(n,p)$, following the framework of~\cite{DKLP1}, consists of three main contiguity arguments. Our starting point is the Poisson cloning model $\GPC(n,p)$ due to~\cite{KimB}, which is contiguous to $\cG(n,p)$ (see \S\ref{sec-prelims}). The first step is to reduce the 2-core of $\GPC(n,p)$ to a random graph with a given (random) degree sequence (Theorem~\ref{thm-Lambda-contiguity} in \S\ref{sec:structure-2-core}). The second step reduces this to a model where a kernel is expanded to a 2-core by subdividing its edges via i.i.d.\ geometric variables (Theorem~\ref{thm-Poisson-contiguity} in \S\ref{sec:poisson-geometric}). The final step handles the attached trees and completes the proof of the main theorem (\S\ref{sec:structure-gc}).

It is already at the first step where the analysis of our previous work~\cite{DKLP1} breaks when $p=\lambda/n$ for fixed $\lambda>1$. Our original approach at this stage relied on showing that a certain stopping time $\taupc$ for a process that produces the 2-core (the so-called COLA algorithm due to Kim) is absolutely continuous w.r.t.\ Lebesgue measure (after normalizing it by its standard deviation). However, crucial in that proof was the fact that the $p=(1+\epsilon)/n$ and $\epsilon=o(1)$, e.g.\ illustrated by the fact that the size of the 2-core given the aforementioned $\taupc$ has a standard deviation smaller than the mean by a factor of $\sqrt{\epsilon}$, and as such is concentrated when $\epsilon\to0$. New arguments were required to establish Theorem~\ref{thm-Lambda-contiguity}, including the use of the powerful Pittel-Wormald~\cite{PW} local limit theorems already in this stage of the proof (cf.~\cite{DKLP1} where this tool was applied only in the second stage of the reductions). Finally, various arguments were simplified, either in places where the dependency in $\epsilon$ would no longer play a role or in situations where the fact that $1/\epsilon=O(1)$ allows direct application of standard local limit theorems.

\section{Preliminaries}\label{sec-prelims}

\subsection{Cores and kernels} The \emph{$k$-core} of a graph $G$, denoted by $G^{(k)}$, is its maximum subgraph $H \subset G$ where
every vertex has degree at least $k$. This subgraph is unique, and can be obtained by
repeatedly deleting any vertex whose degree is smaller than $k$ (in an arbitrary order).
The \emph{kernel} $\K$ of $G$ is obtained by taking its $2$-core
$\TC[G]$ minus its disjoint cycles, then repeatedly contracting any path where all internal vertices have degree-2 (replacing it by a single edge). Notice that, by
definition, the degree of every vertex in $\K$ is at least $3$. At
certain times the notation $\ker(G)$ will be useful to denote a
kernel with respect to some specific graph $G$. Note that $\ker(G)$
is usually different from the 3-core of $G$.

\subsection{Configuration model}\label{sec-prelim-conf}
This model, introduced by Bollob\'{a}s \cite{Bollobas1}, provides a remarkable method for constructing
random graphs with a given degree distribution, which is highly useful to their analysis. We describe
this for the case of random $d$-regular graphs for $d$ fixed (the model is similar for other degree distributions);
see \cites{Bollobas2,JLR,Wormald99} for additional information.

Associate each of the $n$ vertices with $d$ distinct points (also referred to as ``half-edges''), and consider a uniform perfect matching on these points. The random $d$-regular graph is obtained by contracting each cluster of the $d$ points corresponding to a vertex, possibly introducing multiple edges and self-loops. Clearly, on the event that the obtained graph is simple, it is uniformly distributed among all $d$-regular graphs, and furthermore, one can show that this event occurs with probability bounded away from $0$ (namely, with probability about $\exp(\frac{1-d^2}4)$). Hence, every event that occurs w.h.p.\ for this model, also occurs w.h.p.\ for a random $d$-regular graph.

\subsection{Poisson cloning}\label{sec-prelim-pc}
Following is a brief account on the Poisson cloning model $\GPC(n,p)$, introduced in \cites{KimA,KimB}.
Let $V$ be the set of $n$ vertices, and $\Po(\lambda)$ denote a Poisson random variable with mean $\lambda$.
Let $\{d(v)\}_{v \in V}$ be a sequence of i.i.d.\ $\Po(\lambda)$ variables with $\lambda = (n-1)p$. Then, take $d(v)$ copies of each vertex $v \in V$ and the copies of $v$ are called \emph{clones} of $v$ or simply \emph{$v$-clones}. Define $N_\lambda \deq \sum_{v \in V} d(v)$. If $N_\lambda$ is even, the multi-graph $\GPC(n,p)$ is obtained by generating a uniform random perfect matching of those $N_\lambda$ clones (e.g., via the configuration model, where every clone is considered to be a half-edge) and contracting clones of the same vertex. That is to say, each matching of a $v$-clone and a $w$-clone is translated into the edge $(v, w)$ with multiplicity. In the case that $v = w$, it contributes a self-loop with degree $2$. On the other hand, if $N_\lambda$ is odd, we first pick a uniform clone and translate it to a special self-loop contributing degree $1$ of the corresponding vertex (this special self-loop plays no important role in the model and can be neglected throughout the paper). For the remaining clones, generate a perfect matching and contract them as in the $N_\lambda$ even case.

The following theorem of \cite{KimB} states that the Poisson cloning model is \emph{contiguous} with the Erd\H{o}s-R\'{e}nyi model. Hence, it suffices to study Poisson cloning model in order to establish properties of the Erd\H{o}s-R\'{e}nyi model.

\begin{theorem}[\cite{KimB}*{Theorem 1.1}]\label{thm-poisson-ER}
Suppose $p \asymp 1/n$. Then there exist constants $c_1, c_2 >0$ such that for any collection $\cF$ of simple graphs, we have
\[ c_1\P(\cG_{\textsc{pc}}(n,p) \in \cF) \leq \P(\cG(n,p) \in \cF) \leq c_2\big(\P(\GPC(n,p) \in \cF)\big)^{1/2} + \mathrm{e}^{-n}\big)\,.\]
\end{theorem}

Note that as $p = \lambda/n$ for $\lambda>1$ fixed we may clearly replace the rate $\lambda=(n-1)p$ in the Poisson-cloning model definition simply by $\lambda=np$.

\subsection{Local limit theorem}
Throughout the proofs we will need to
establish local limit theorems for various parameters in the graph.
To this end, we will repeatedly apply the following special case of
a result in~\cite{Durrett}.
\begin{theorem}[\cite{Durrett}*{Ch.\ 2, Theorem 5.2},
reformulated]\label{thm-local} Let $X$ be a random variable on $\N$
with $\P(X = k) > 0$ for all $k\in\N$. Suppose that $\E X = \nu<
\infty$ and $\var X = \sigma^2 < \infty$. Let $X_i$ be i.i.d
distributed as $X$ and $S_m = \sum_{i=1}^m X_i$. Then as $m \to
\infty$, we have
\[\sup_{x \in \mathcal{L}_m}\left|\sqrt{m}\,\P \left(\frac{S_m -  m\nu}{\sqrt{m}} = x\right) - \frac{1}{\sigma\sqrt{2\pi}} \mathrm{e}^{-x^2/\sigma^2}\right| \to 0~,\]
where $\mathcal{L}_m = \{(z - m \nu )/ \sqrt{m}: z\in \Z\}$.
\end{theorem}

\section{The 2-core of Poisson cloning}\label{sec:structure-2-core}

By Theorem~\ref{thm-poisson-ER}, the random graph $\cG(n,p)$ in our range of parameters is contiguous to the Poisson cloning model, where every vertex gets an i.i.d.\ $\Po(\lambda)$ number of half-edges (clones) and the final multigraph is obtained via the configuration model. In this section we will reduce the 2-core of the supercritical Poisson cloning model to a more tractable model --- a random graph uniformly chosen over all graphs with a given degree sequence.

\begin{definition}[Poisson-configuration model with parameters $n$ and $\lambda$]\label{def-poisson-conf-model}\mbox{}
\begin{enumerate}[(1)]
\item Let $\Lambda\sim \mathcal{N}\left(\lambda - \mu, 1/n\right)$, consider $n$ vertices and assign an independent variable $D_u \sim \Po(\Lambda)$ to each vertex $u$. Let $N_k = \#\{u : D_u = k\}$ and $N = \sum_{k\geq 2}N_k$.
\item Construct a random multigraph on $N$ vertices, uniformly chosen over all graphs with $N_k$ degree-$k$ vertices for $k\geq 2$ (if $N$ is odd, choose a vertex $u$ with $D_u=k \geq 2$ with probability proportional to $k$, and give it $k-1$ half-edges and a self-loop).
\end{enumerate}
\end{definition}

\begin{theorem}\label{thm-Lambda-contiguity}
Let $G \sim \GPC(n,p)$ be generated by the Poisson cloning model for
$p=\lambda/n$, where $\lambda > 1$ is fixed.
Let $\TC[G]$ be its $2$-core, and $H$
be generated by the Poisson-configuration model corresponding to $n,p$. Then for any set of graphs $\mathcal{A}$ such that $\P(H \in
\mathcal{A}) \to 0$, we have $\P(\TC[G] \in \mathcal{A}) \to 0$.
\end{theorem}

To prove the above Theorem~\ref{thm-Lambda-contiguity} we outline a specific way to generate $\GPC(n,p)$ due to~\cite{KimB}. Let $V$ be a set of $n$ vertices and consider $n$ horizontal line segments ranging from $(0, j)$ to $(\lambda, j)$, for $j=1, \ldots, n$ in $\R^2$. Assign a Poisson point process with rate $1$ on each line segment independently. Each point $(x, v)$ in these processes is referred to as a $v$-clone with the assigned number $x$. The entire set of Poisson point processes is called a Poisson $\lambda$-cell.

Given the Poisson $\lambda$-cell, various schemes can be used to generate a perfect matching on all points. One such way is the ``Cut-Off Line Algorithm'' (COLA), defined in \cite{KimA}, which is useful in finding the 2-core $\TC[G]$, described as follows.
The algorithm maintains the position of a ``cut-off line'', a vertical line in $\R^2$ whose initial $x$-coordinate equals $\lambda$, and gradually moves leftwards. In the beginning the line is positioned at $\lambda$, and as the line progresses it matches previously unmatched clones.
To describe the stopping rule of the algorithm, we need the following definitions. At any given point, we call a vertex $v \in V$ (and its unmatched clones) \emph{light} if it has at most one unmatched clone (and \emph{heavy} otherwise). At the beginning of the process, all the light clones are placed in a stack. The order by which these clones are inserted into the stack can be arbitrary, as long as it is oblivious of the coordinates assigned to the clones.
Define $\taupc$ to be the $x$-coordinate of the cut-off line once the algorithm terminates, i.e., at the first time when there are no light clones. We will argue that $\taupc$ is concentrated about
$\lambda-\mu$ with a standard deviation of $1/\sqrt{n}$, yet before doing so we explain its role in determining the structure of the 2-core of the graph. The above algorithm repeatedly matches light clones until all of them are exhausted --- precisely as the cut-off line reaches $\taupc$. As stated in \S\ref{sec-prelims}, the $2$-core of a graph can be obtained by repeatedly removing vertices of degree at most $1$ (at any arbitrary order), thus it is precisely comprised of all the unmatched clones at the moment we reach $\taupc$.

\begin{algorithm}
\[
\fbox{\parbox{.98\linewidth}{\mdseries

\begin{enumerate}[1.]
\item \label{it-cola-1} Let $(x,u)$ be the first clone in the stack. Move the cut-off line leftwards until it hits an unmatched clone $(y,v) \neq (x,u)$.
\item Remove $(x,u)$ from the stack, as well as $(y,v)$ (if it is there).
\item Match $(x,u)$ and $(y,v)$ and re-evaluate $u$ and $v$ as light/heavy.
\item Add any clone that just became \emph{light} into the stack.
\item If the stack is nonempty return to Step~\ref{it-cola-1}, otherwise quit and denote the stopping time by $\taupc$ (final $x$-coordinate of cut-off line).
\end{enumerate}
}}
\]
\caption{\textsc{Cut-Off Line Algorithm}}
\label{algorithm-cola}
\end{algorithm}

The following theorem establishes concentration for $\taupc$. Its proof will follow from known estimates on the concentration of $|\TC|$ in $\cG(n,p)$.
\begin{theorem}[Upper bound on the window of $\taupc$]\label{thm-Lambda-C-upper}
There exist constants $C,c > 0$ so that for all
$\gamma>0$ with $\gamma = o\big(\sqrt{n}\big)$, the following holds:
\begin{equation}\label{eq-Lambda-C-window}
\P\left(|\taupc - (\lambda-\mu)| \geq \gamma/ \sqrt{ n} \right) \leq
C \mathrm{e}^{-c \gamma^2}\,.
\end{equation}
\end{theorem}
\begin{proof}
It is well known that in the super-critical random graph $\cG(n,\lambda/n)$ with $\lambda>1$ fixed w.h.p.\ all components except one (the giant) are trees or unicyclic (see e.g.~\cite{JLR}*{Theorem~5.12}) and in addition the total expected number of vertices which belong to unicyclic components is bounded (see e.g.~\cite{Bollobas2}*{Theorem~5.23}). In particular this implies that the $2$-core of $\cG(n,\lambda/n)$ for $\lambda>1$ fixed consists of $\TC$, the 2-core of the giant component, plus disjoint cycles whose total number of vertices is w.h.p.\ at most, say, $O(\log n)$.

A special case of a powerful result of Pittel and Wormald~\cite{PW} (the full statement of this theorem appears later as Theorem~\ref{thm-PW-local-limit}) implies that \[\E|\TC|=(1-\mu)(1-\tfrac{\mu}{\lambda})n\] and in addition $(|\TC|-\E|\TC|)/\sqrt{n}$ is in the limit Gaussian with variance of $O(1)$.
Combining these facts, there exists some fixed $c > 0$ so that a random graph $F \sim \cG(n,p)$ in our regime has a 2-core $\TC[F]$ whose size satisfies
\[ \P\left(\big||\TC[F]| - (1-\mu)(1-\tfrac{\mu}{\lambda})n|\big| \geq \gamma/\sqrt{n}\right) \leq \exp(-c\gamma^2) \,.\]
By Theorem~\ref{thm-poisson-ER} it then follows that for $G \sim \cG_{\textsc{pc}}(n,p)$,
\[ \P\left(\big||\TC[G]| - (1-\mu)(1-\tfrac{\mu}{\lambda})n|\big| \geq \gamma/\sqrt{n}\right) \leq C \exp(-c\gamma^2)\,,\]
where $C = 1/ c_1$ from that theorem.

To conclude the proof, observe on the event $\taupc = x$, the size of $\TC[G]$ is binomial with parameters $\Bin(n,p_2^+(x))$ where
\[ p^+_2(x) \deq \sum_{k\geq 2} \mathrm{e}^{-x} \frac{x^k}{k!} = 1 - \mathrm{e}^{-x} - x \mathrm{e}^{-x}\,.\]
It is easy to verify that $x=\lambda - \mu$ is the unique positive solution of $p^+_2(x) = (1-\mu)(1-\tfrac{\mu}{\lambda})$.
This function further has $\frac{d}{dx} p^+_2(x) = x e^{-x}$ thus its derivative is uniformly bounded away from 0 in the interval
  $[\tfrac12(\lambda-\mu),2(\lambda-\mu)]$. In particular, shifting $\taupc$ by $\gamma/\sqrt{n}$ would shift the mean of the above
  variable by order $\gamma/\sqrt{n}$ and the desired result
  follows. 
\end{proof}
The above theorem established the concentration of $\taupc$ and as such reduced the Poisson-cloning model to the Poisson-configuration model given the event $\taupc = \lambda-\mu+o(1)$. With this in mind, the argument in the proof above stating that the disjoint cycles outside the giant component have bounded expectation in $\cG(n,p)$, along with the contiguity between $\cG(n,p)$ and Poisson-cloning, now immediately yield the following:
\begin{corollary}\label{cor-cycles}
Let $H$ be generated by the Poisson-configuration model given $\Lambda = \ell$, where $\ell=\lambda-\mu+o(1)$.
Define $H'$ as the graph obtained by deleting every disjoint cycle from $H$.
Let $N_2$ be the number of vertices with degree $2$ in $H$, and $N'_2$ be the corresponding quantity for $H'$.
Then $N'_2 = N_2 + O_\mathrm{P} (1)$.
\end{corollary}

\subsection{Contiguity of Poisson-cloning and Poisson-configuration}

A key part of showing the contiguity result is the following lemma which controls the edge distribution in the Poisson-configuration model.

\begin{lemma}\label{lem-P(K|Lambda)-ratio}
Let $N_k$ denote the number of degree-$k$ vertices in the Poisson-configuration model,
and set $\Lambda_0 = \lambda - \mu$.
For any fixed $M > 0$ there exist some $c_1,c_2>0$ such that the following holds:
If $n_3,n_4,\ldots$ satisfy
\begin{align*}
\begin{array}{rl}
\Big|n\left(1-\mathrm{e}^{-\Lambda_0}(1+\Lambda_0 + \tfrac{\Lambda_0^2}2)\right) - \sum_{k\geq3}n_k \Big| &\leq M\sqrt{n}\,,\\
\Big|n\Lambda_0\left(1-\mathrm{e}^{-\Lambda_0}(1 + \Lambda_0)\right) - \sum_{k\geq3}k n_k \Big| &\leq M\sqrt{n}\,
\end{array}
\end{align*}
and $x$ satisfies $|x - \Lambda_0| \leq M/\sqrt{n}$ then
\[
c_1 \leq \frac{\P\left(N_k = n_k~\mbox{ for all }k\geq 3 \given \Lambda = x\right)}
{\P\left(N_k = n_k~\mbox{ for all }k\geq 3 \given \Lambda = \Lambda_0\right)} \leq c_2\,.
\]
\end{lemma}
\begin{proof}
Throughout the proof of the lemma, the implicit constants in the $O(\cdot)$ notation depend on $M$.
Write $m = \sum_{k \geq 3}n_k$ and $r=\sum_{k\geq 3}k n_k$, and let
$A=A(n_3,n_4,\ldots)$ denote the event $\{N_k = n_k\mbox{ for all }k\geq 3\}$.
Setting
\[ \Xi= \Xi(x,\Lambda_0) \deq \frac{\P\left(A \given \Lambda = x\right)}
{\P\left(A \given \Lambda = \Lambda_0\right)} \,,\]
we are interested in uniform bounds for $\Xi$ from above and below. Let
 \[p_k(x) = \P(\Po(x)=k) = \mathrm{e}^{-x}x^k/k!~,~\mbox{ and }p^-_k(x) = \P(\Po(x) \leq k)\,.\]
 and observe that
\begin{align*}
\Xi &=
\bigg(\frac{p^-_2(x)}{p^-_2(\Lambda_0)}\bigg)^{n-m} \prod_k \bigg(\frac{p_k(x)}{p_k(\Lambda_0)}\bigg)^{n_k} = \mathrm{e}^{-n(x-\Lambda_0)}\bigg(\frac{1+x+\frac{x^2}2}{1+\Lambda_0+\frac{\Lambda_0^2}2}\bigg)^{n-m} \Big(\frac{x}{\Lambda_0}\Big)^r\,,
\end{align*}
and so
\begin{align*}
\log \Xi &=
n(\Lambda_0-x) + (n-m)\log\bigg(\frac{1+x+\tfrac{x^2}2}{1+\Lambda_0+\frac{\Lambda_0^2}2}\bigg)+ r\log\frac{x}{\Lambda_0}\,.
\end{align*}
Using Taylor's expansion and recalling that $x -\Lambda_0 = O(1/\sqrt{n}) $,
\begin{align*}
\log\bigg(\frac{1+x+\tfrac{x^2}2}{1+\Lambda_0+\frac{\Lambda_0^2}2}\bigg)
 &= \frac{1+\Lambda_0}{1+\Lambda_0+\frac{\Lambda_0^2}2}(x-\Lambda_0) +O(1/n)\,,
\end{align*}
and we
deduce that
\begin{align*}
\log\Xi &=
n(\Lambda_0-x) +  (n-m) \frac{1+\Lambda_0}{1+\Lambda_0+\frac{\Lambda_0^2}2}(x-\Lambda_0)
+ r \frac{x - \Lambda_0}{\Lambda_0} + O(1)\,.
\end{align*}
Our assumptions on $m,r$ now yield that
\begin{align*}
\log\Xi &=
n(\Lambda_0-x) +  n \mathrm{e}^{-\Lambda_0} (1+\Lambda_0)(x-\Lambda_0)\\
&+ n \left(1-\mathrm{e}^{-\Lambda_0}(1+\Lambda_0)\right)(x-\Lambda_0) + O(1) = O(1)\,,
\end{align*}
completing the proof.
\end{proof}

Using the above estimate we are now able to conclude the main result of this section, which reduces the $2$-core of Poisson-cloning to the graph generated by the Poisson-configuration model.

\begin{proof}[\textbf{\emph{Proof of Theorem~\ref{thm-Lambda-contiguity}}}]
Recall that $H$ is the random graph generated by the Poisson-configuration model, and $\TC[G]$ is the 2-core of Poisson-cloning.

Fix $\delta > 0$, and with the statement of Theorem~\ref{thm-Lambda-C-upper} in mind,
as well as the definition of $\Lambda$ as Gaussian with parameters $\mathcal{N}\left(\lambda - \mu, 1/n\right)$,
set
\[B = (\lambda-\mu - M/\sqrt{ n},~ \lambda-\mu + M/\sqrt{n})\,,\] where $M=M(\delta)$ is a sufficiently large constant such that
\[ \P(\taupc \in B\,,\,\Lambda \in B) \geq 1 - \delta\,.\]

Following the notation of Lemma~\ref{lem-P(K|Lambda)-ratio}, let $N_k$ be the number of degree-$k$ vertices in $H$.
Conditioned on $\Lambda = x$ we have that $\sum_{k\geq 3} N_k = \sum_{u} \one_{\{D_u \geq 3\}}$ where the $D_u$ are i.i.d.\ $\Po(x)$, hence the Central Limit Theorem implies that $\sum_{k\geq 3} N_k$ is concentrated around $n\big(1-\mathrm{e}^{-x}(1+x + \tfrac{x^2}2)\big)$ with a window of $O(\sqrt{n})$. A similar statement holds for $\sum_{k\geq 3} k N_k = \sum_u D_u \one_{\{D_u \geq 3\}}$.
Since $\Lambda$ is Gaussian with variance $1/n$, removing the conditioning on $\Lambda$ introduces $O(1/\sqrt{n})$ fluctuations which translate to $O(\sqrt{n})$ fluctuations in the above mentioned random variables.
Altogether, if $M(\delta)>0$ is large enough then each of the variables $\sum_{k\geq 3}N_k$ and $\sum_{k\geq 3}k N_k$ is within $M\sqrt{n}$ of its mean, except with probability $\delta$.

Let $\Gamma$ denote the set of all sequences $\{n_k : k\geq 3\}$ which satisfy the assumptions of Lemma~\ref{lem-P(K|Lambda)-ratio}. By the above discussion, we can restrict our attention to degree-sequences in $\Gamma$ at a cost of events whose probability is at most $\delta$.
By the conclusion of that lemma, the probability to have $\{N_k : k \geq 3\} \in \Gamma$ given $\Lambda=x$ for some $x\in B$ is uniformly bounded below and above by the corresponding probability given $\Lambda = \lambda - \mu$.
In particular, for any $x\in B$ the event that $H$ is in $\mathcal{A} \cap \big\{ \{N_k:k\geq 3\}\in \Gamma\big\}$ given $\Lambda = x$ has up to constants (that depend on $\delta$) the same probability for any $x$ in this region.

Further fix $\delta'>0$ and define
\[ D \deq \left\{x \in B: \P\left(H\in \mathcal{A}\,,\,\{N_k(H):k\geq 3\}\in \Gamma
\mid \Lambda = x\right) \geq \delta'\right\}\,.\]
We claim that for any large enough $n$ we have $D = \emptyset$. Indeed, this follows immediately from the discussion above, since the existence of some $x\in D$ would imply that in particular
\[ \P\left(H\in \mathcal{A}\,,\,\{N_k(H):k\geq 3\}\in \Gamma
\mid \Lambda \in B\right) \geq \delta' / c(\delta)\,,\]
which, since $\P(\Lambda \in B) \geq 1-\delta$, contradicts the fact that $\P(H\in \mathcal{A}) = o(1)$.

With $D = \emptyset$, reiterating the argument on the uniformity for on $x \in B$
of the above probability given $\Lambda = x$ we deduce that
for any $x \in B$,
\[ \P\left(H\in \mathcal{A}\,,\,\{N_k(H):k\geq 3\}\in \Gamma
\mid \Lambda = x\right) \leq \delta'\,.\]
To complete the proof, observe that the Poisson-configuration model given $\Lambda=x$ is equivalent to the Poisson-cloning model given $\taupc=x$. Therefore, combining the arguments thus far we arrive at the following estimate, valid for any $x\in B$:
\[ \P\big(\TC[G]\in \mathcal{A} \mid \taupc = x\big) =
\P\left(H \in \mathcal{A} \mid \Lambda = x\right) \leq \delta' + \delta\,,\]
where the $\delta$-term accounted for the probability of $\{N_k(H):k\geq 3\}\notin \Gamma$.
The proof is concluded by recalling that $\taupc \in B$ except with probability $\delta$, and letting $\delta'\to 0$ followed by $\delta\to0$.
 \end{proof}

\section{Constructing the 2-core of the random graph}\label{sec:poisson-geometric}

So far we have reduced the 2-core of Poisson-cloning to the simpler Poisson-configuration model as given in Definition~\ref{def-poisson-conf-model}.
In this section we will reduce the Poisson-configuration model to the Poisson-geometric model, defined as follows.
Recall that $\mu<1$ is the conjugate of $\lambda>1$ as defined in Theorem~\ref{mainthm-struct}.
\begin{definition}[Poisson-geometric model for $n$ and $p=\lambda/n$]\label{def-poisson-geo-model}\mbox{}
\begin{enumerate}[(1)]
\item Let $\Lambda\sim \mathcal{N}\left(\lambda - \mu, \frac1{n}\right)$ and assign an independent $\Po(\Lambda)$ variable $D_u$ to each vertex $u$. Let $N_k = \#\{u : D_u = k\}$ and $N = \sum_{k\geq 3}N_k$.
\item Construct a random multigraph $\K$ on $N$ vertices, uniformly chosen over all graphs with $N_k$ degree-$k$ vertices for $k\geq 3$ (if $\sum_{k\geq 3}k N_k$ is odd, choose a vertex $u$ with $D_u=k \geq 3$ with probability proportional to $k$, and give it $k-1$ half-edges and a self-loop).
\item Replace the edges of $\K$ by paths of length i.i.d.\ $\Geom(1-\mu)$.
\end{enumerate}
\end{definition}

\begin{theorem}\label{thm-Poisson-contiguity}
Let $H$ be generated by the Poisson-configuration model w.r.t.\ $n$ and
$p=\lambda/n$ where $\lambda>1$ is fixed.
Let $\tilde{H}$ be generated by the Poisson-geometric model corresponding to $n,p$. Then for any set of graphs $\mathcal{A}$ such that $\P(\tilde{H} \in
\mathcal{A}) \to 0$, we have $\P(H \in \mathcal{A}) \to 0$.
\end{theorem}

Both models clearly share the same kernel and only differ in the way this kernel is then expanded to
form the entire graph (replacing edges by paths). To prove the above theorem we need to estimate the distribution of
the total number of edges in each model and show they are in fact contiguous.
A first step towards this goal is to control the number of
edges in each model. Fix some large $M
> 0$, and let $\cB_M$ denote the following set of ``good'' kernels:
\begin{align}\label{eq-good-kernels}
\cB_M \deq \left\{\K :\begin{array}{c}
\Big||\K| - n\left(1-\mathrm{e}^{-\Lambda_0}(1+\Lambda_0 + \tfrac{\Lambda_0^2}2)\right) \Big| \leq M\sqrt{ n}\\
\Big||E(\K)| - \tfrac12 n\Lambda_0\left(1-\mathrm{e}^{-\Lambda_0}(1 + \Lambda_0)\right) \Big| \leq M\sqrt{ n}
\end{array}\right\}\,,
\end{align}
where $\Lambda_0 = \lambda - \mu$.
The next lemma estimates the number of edges in the Poisson configuration
model given that the kernel belongs to $\cB_M$ as above.
\begin{lemma}\label{lem-poisson-conf-upper}
Define $M>0$, $I_M, \mathcal{B}_M$ as above and set $\Lambda_0=\lambda-\mu$. Let $H$ be generated by the
Poisson-configuration model. There exists some constant $c(M) > 0$ so that
for any $\K \in \mathcal{B}_M$ and $s$ with $\big|s - \frac{n}2\left(\Lambda_0 - \mathrm{e}^{-\Lambda_0}\Lambda_0 \right)\big| \leq M\sqrt{ n}$,
\[\P\left(|E(H)| = s \,,\, \Lambda \in I_M \given \ker(H) = \K\right) \leq \frac{c}{\sqrt{ n}}\,.\]
\end{lemma}
\begin{proof}
Let $x \in I_M$ and $\K \in \cB_M$, and write $m = |\K|$ and $r=|E(\K)|$ for the number of vertices and the edges in the kernel respectively. We will first estimate $\P(|E(H)| = s \given \Lambda = x\,,\, \ker(H) = \K)$, and the required inequality will then
readily follow from an integration over $x \in I_M$.

Note that, given $\Lambda = x$ and $\ker(H)=\K$, the number of edges in $H$ is the $r$ edges of $\K$ plus an added edge for each degree $2$ variable out of the $n-m$ variables (i.i.d.\ $\Po(x)$) that have $\{ u : D_u \leq 2\}$. That is, in this case
\[
|E(H)| \sim r+\Bin\left( n-m, \frac{x^2/2}{1 + x + x^2/2}\right)\,.
\]
It is straightforward to verify that for $x \in I_M$ and $ \K \in
\cB_M$, we have \begin{align*}\E \big(|E(H)| \;\big|\; \Lambda = x, \ker(H) =
\K\big)& = r + (n-m)\frac{x^2/2}{1 + x + x^2/2}\\
& = \frac{n}2\left(\Lambda_0 -
\mathrm{e}^{-\Lambda_0}\Lambda_0\right) + O(\sqrt{n}) = s+
O(\sqrt{n})\,.\end{align*}
The required estimate now follows
immediately from Theorem~\ref{thm-local}.
\end{proof}
We now turn to analyze the number of edges in Poisson-Geometric
model.
\begin{lemma}\label{lem-poisson-geo-lower}
Let $M>0$ and $\mathcal{B}_M$ be as in \eqref{eq-good-kernels}. Let $\tilde{H}$ be generated by the Poisson-geometric model. There exists some constant $c = c(M) > 0$ so that
for any $\K \in \mathcal{B}_M$ and $s$ with $\big|s - \frac{n}2(\lambda-\mu)\left(1-\frac{\mu}{\lambda}\right)\big| \leq M\sqrt{ n}$,
\[\P\left(|E(\tilde{H})| = s \given \ker(\tilde{H}) = \K\right) \geq \frac{c}{\sqrt{n}}\,.\]
\end{lemma}
\begin{proof}
By definition, given that $\ker(\tilde{H})=\K$, the variable $|E(\tilde{H})|$ is the sum of $|E(\K)|$ i.i.d.\ geometric variables
with mean $1/(1-\mu)$.

Denote by $r$ the number of edges in the kernel $\K$, and let $s$ be a candidate for the number of edges in the expanded 2-core $\tilde{H}$. As stated in the lemma (recall definition \eqref{eq-good-kernels}), we are interested in the following range for $r,s$:
\begin{align*}
r &= \frac{n}{2}(\lambda - \mu) (1 - \tfrac{\mu}{\lambda})(1 - \mu) +
c_1 \sqrt{ n}, &\quad (|c_1| \leq M)\,,\\
s & = \frac{n}{2}(\lambda - \mu) (1 - \tfrac{\mu}{\lambda}) + c_2 \sqrt{n}, &\quad (|c_2| \leq M)\,.
\end{align*}
It is clear that given $\ker(\tilde{H}) = \K$ with $|E(\K)| = r$,
the number of edges in $\tilde{H}$ is distributed as $\sum_{i=1}^r
X_i$, where the $X_i$'s are independent geometric random variables with
mean $\frac{1}{1-\mu}$, i.e., $\P(X_i=k) = \mu^{k-1}(1-\mu)$ for $k=1,2,\ldots$. Since $s = \frac{r}{1-\mu} + O(\sqrt{r})$, the
desired estimate now follows from
Theorem~\ref{thm-local}.
\end{proof}

We now establish the main result of this section, Theorem~\ref{thm-Poisson-contiguity}, reducing the Poisson-configuration model to the Poisson-geometric model.
\begin{proof}[\emph{\textbf{Proof of Theorem~\ref{thm-Poisson-contiguity}}}]
For some constant $M > 0$ to be specified later, define the event
\begin{align*}
A_M \deq \left\{~\Lambda \in I_M,~ \ker(H) \in \cB_M,~ \left| |E(H)| - \tfrac{n}{2}(\lambda - \mu) (1 - \tfrac{\mu}{\lambda})\right| \leq M \sqrt{n}~
   \right\}\,.
\end{align*}
Fix $\delta > 0$. We claim that for a sufficiently large
$M = M(\delta)$ we have $\P(A_M) \geq 1- \delta$.
To see this, note the following:
\begin{enumerate}[\indent 1.]
  \item In the Poisson-configuration model, $\Lambda \sim \mathcal{N}\left(\lambda-\mu , \frac{1}{ n}\right)$, and $I_M$ includes at least $M$ standard deviations about its mean.
      \item Each of the variables
      $|\K |$ and $E(\K)$ is a sum of i.i.d.\ random variables with variance $O(n)$ and mean as specified in the definition of $\cB_M$, hence their concentration follows from the CLT.
      \item Finally, $E(H)$ is again a sum of i.i.d.\ variables and has variance $O(n)$, only here we must subtract the vertices that comprise disjoint cycles. By Corollary~\ref{cor-cycles}, the number of such vertices is $O_\mathrm{P}(1)$, which is negligible  compared to the $O(\sqrt{n})$ standard deviation of $E(H)$.
\end{enumerate}
Given an integer $s$ and a kernel $\K$, let $\mathcal{D}_{s,\K}$ denote every possible $2$-core with $s$ edges and kernel $\K$.
Crucially, the distribution of the Poisson-configuration model given $E(H)=s$ and $\ker(H)=\K$ is uniform over $\mathcal{D}_{s,\K}$, and so is the Poisson-geometric model given $E(\tilde{H})=s$ and $\ker(\tilde{H})=\K$.
Therefore, for any graph $D \in \mathcal{D}_{s,\K}$,
\begin{align*}
\frac{\P(H = D\given \ker(H)=\K )}{\P(\tilde{H} = D\given \ker(\tilde{H})=\K )}
= \frac{\P (|E(H)| = s \given \ker(H)=\K )}{\P (|E(\tilde{H})| = s \given \ker(\tilde{H})=\K)}\,.
\end{align*}
Combining Lemmas~\ref{lem-poisson-conf-upper} and~\ref{lem-poisson-geo-lower} we get that for some $c=c(M) > 0$,
\[
\frac{\P (|E(H)| = s~,~A_M \given \ker(H)=\K )}{\P (|E(\tilde{H})| = s \given \ker(\tilde{H})=\K)}
\leq c\,.\]
Recalling that $\P(A_M) \geq 1 -\delta$ and letting $\delta \to 0$, we deduce that for any family of graphs $\mathcal{A}$,
if $\P(\tilde{H} \in \mathcal{A}) \to 0$ then also $\P(H \in \mathcal{A}) \to 0$.
\end{proof}

\section{Constructing the giant component}\label{sec:structure-gc}

We begin by analyzing the trees that are attached to $\TC[G]$, the 2-core of $G$.
As before, $\mu < 1$ is the conjugate of $\lambda>1 $ as defined in Theorem~\ref{mainthm-struct}.

%

\begin{proof}[\emph{\textbf{Proof of Theorem~\ref{mainthm-struct}}}]
Write PGW($\mu$)-tree for a Poisson($\mu$)-Galton-Watson tree, for brevity.
Let $\hGC$ denote the graph obtained as follows:
\begin{itemize}
  \item Let $H$ be a copy of $\TC$ (the 2-core of the giant component of $G$).
  \item For each $v \in H$, attach an independent PGW($\mu$) tree rooted at $v$.
\end{itemize}
By this definition, $\GC$ and $\hGC$ share the same 2-core $H$. For simplicity, we will refer directly to $H$ as the 2-core of the model, whenever the context of either $\GC$ or $\hGC$ is clear. We first establish the contiguity of $\GC$ and $\hGC$.

Define the trees decorating the 2-core of $\GC$ as follows:
\[T_u \deq \{v\in \GC: v \mbox{ is connected to } u \mbox{ in } \GC \setminus H\}~\mbox{ for $u\in H$}\,.\]
Clearly, each $T_u$ is a tree as it is connected and has no cycles (its vertices were not included in the 2-core).
We go from $H$ to $\GC$ by attaching the tree $T_u$ to each vertex $u\in H$ (while identifying the root of $T_u$ with $u$).
Similarly, let $\{\tilde{T}_u\}_{u\in H}$ be the analogous trees in $\hGC$.

We next introduce notations for the labeled and unlabeled trees as well as their distributions. For $t \in \N$, let $\mathcal{R}_t$ be the set of all labeled rooted trees on the vertex set $[t]\deq\{1,\ldots,t\}$, and let $U_t$ be chosen uniformly at random from $\mathcal{R}_t$. For $T \in \mathcal{R}_t$ and a bijection $\phi$ on $[t]$, let $\phi(T)$ be the tree obtained by relabeling the vertices in $T$ according to $\phi$. Further let $T'$
be the corresponding rooted unlabeled tree: $T' \deq \{\phi(T): \phi \mbox{ is a bijection on }[t]\}$.

Let $\{t_u : u \in H\}$ be some integers. Given that $\{\,|T_u| = t_u \mbox{ for all } u\in H\,\}$,
we know by definition of $\cG(n,p)$ that $T_u$ is independently and uniformly distributed among all labeled trees of size $t_u$ rooted at $u$. In particular, given this event each $T'_u$ is independently distributed as $U'_{t_u}$ (the unlabeled counterparts of $T_u$ and $U_{t_u}$).
On the other hand, Aldous~\cite{Aldous} (see~\cite{AP}) observed that if $T$ is a PGW-tree then $T'$ has the same distribution as $U'_t$ on the event $\{|T| = t\}$. Thus, conditioned on the event $\{\,|\tilde{T}_u| = t_u \mbox{ for all } u\in H\,\}$
we also get that $\tilde{T}'_k$ has the same distribution as $U'_{t_k}$.

We now turn to the sizes of the attached trees in $\GC$ and $\hGC$. Letting $\{t_u : u \in H\}$ be some integers and writing
$N = \sum_{u \in H} t_u\,,$
we claim that by definition of $\cG(n,p)$ every extension of the 2-core $H$ to the component $\GC$, using trees whose sizes sum up to $N$, has the same probability. To see this argue as follows. Fix $H$ and notice that the probability of obtaining a component with a 2-core is $H$ and an extension $X$ connecting it to $N-|H|$ additional vertices only depends on the number of edges in $H$ and $X$ (and the fact that this is a legal configuration, i.e., $H$ is a valid 2-core and $X$ is comprised of trees). Therefore, upon conditioning on $H$ the probabilities of the various extensions $X$ remain all equal.
Cayley's formula gives that there are $m^{m-1}$ labeled rooted trees on $m$ vertices, and so,
\begin{align}\label{eq-size-tree-distr}
\P&\left(|T_u| = t_u \mbox{ for all } u \in H \given H \right)
= \P \left(|\GC| = N \given H \right) \frac{1}{Z(N)}\frac{N!}{\prod_{u\in H} t_u !} \prod_{u\in H} t_u^{t_u -1}\nonumber\\
&= \P (|\GC| = N \given H)\frac{1}{Z'(N)} \prod_{u \in H} \Big[\frac{t_u^{t_u -1}}{\mu t_u!} (\mu \mathrm{e}^{-\mu})^{t_u}\Big],
\end{align}
where $Z(N)$ and $Z'(N)$ are the following normalizing constants
\begin{align*}
Z'(N) &= \sum_{\{r_u\}:\sum_{u\in H}r_u = N} \prod_{u \in H} \Big[\frac{r_u^{r_u -1}}{\mu r_u!} (\mu \mathrm{e}^{-\mu})^{r_u}\Big]\,,\\
 Z(N) &= Z'(N) \mu^{N-|H|} \mathrm{e}^{-\mu N}\,.
\end{align*}
It is well-known that the size of a Poisson$(\gamma)$-Galton-Watson tree $T$ follows a Borel($\gamma$) distribution, namely
\begin{equation}\label{eq-PGW-size}
\P(|T| = t)=\frac{t^{t-1}}{\gamma t!}(\gamma \mathrm{e}^{-\gamma})^{t}\,.
\end{equation}
Recalling that $\tilde{T}_u$ are independent PGW($\mu$)-trees, it follows that
\[Z'(N) = \sum_{\{r_u\}:\sum_{u\in H}r_u = N} \Big[\prod_{u\in H}\P(|\tilde{T}_u| = r_u) \Big]= \P\left(|\hGC| = N \given H\right)\,.\]
Combining this with \eqref{eq-size-tree-distr} and \eqref{eq-PGW-size}, we obtain that
\begin{align}\label{eq-prob-ratio}
\frac{\P(|T_u| = t_u \mbox{ for all } u \in H \given H )}{\P(|\tilde{T}_u| = t_u \mbox{ for all } u \in H \given H )}
= \frac{\P (|\GC| = N \given H)}{\P (|\hGC| = N \given H)}\,.
\end{align}
At this point, we wish to estimate the ratio in the right hand side above. To this end, we need the following result of~\cite{PW}.
\begin{theorem}[\cite{PW}*{Theorem 6}, reformulated]\label{thm-PW-local-limit}
Let $\mathbf{b}(\lambda) = \left(\begin{smallmatrix} b_1(\lambda) \\ b_2(\lambda) \\ b_3(\lambda)\end{smallmatrix}\right)$
where
\begin{align*}
b_1(\lambda) = (1 - \mu)\left(1 - \tfrac{\mu}\lambda\right)\,,\, b_2(\lambda) = \mu \left(1 - \tfrac{\mu}\lambda\right)\,,\,b_3(\lambda) = \tfrac12\left(1 - \tfrac\mu\lambda\right)(\lambda + \mu -2)\,.
\end{align*} There exist positive definite matrices $K_p(\lambda)$ and $K_m(\lambda)$ such that
\begin{enumerate}[(i)]
\item \label{item-PW-central-limit}$(|H|, |\GC| - |H|, |E(H)| - |H|)$ is in the limit Gaussian with a mean vector $n \mathbf{b}$ and a covariance matrix $n K_p$.
\item \label{item-PW-local-limit} If $A_m \deq K_m^{-1}$ and $B$ denotes the event that $|E(G)|=m$ for some $m= (1+o(1))\lambda n/2$, and there is a unique component of size between $\frac12(1-\frac{\mu}{\lambda}) n$ and $2(1-\frac{\mu}{\lambda} )n$ and none larger, then
\begin{align}
\P& \left(|H| = n_1, |\GC| - |H| = n_2, |E(H)| - |H| = n_3 \given B\right)\nonumber \\
&= \frac{\sqrt{\det(A_m)}+o(1)}{(2\pi n)^{3/2}} \exp\left(-\tfrac{1}{2}\mathbf{x}^T A_m \mathbf{x}\right)\,,
\end{align}
uniformly for all $(n_1, n_2, n_3) \in \N^3$ such that
\[(K_p(1, 1)^{-1/2} x_1,K_p(2, 2)^{-1/2} x_2, K_p(3, 3)^{-1/2} x_3 )\]
is bounded, where $\mathbf{x}^T=(x_1,x_2,x_3)$ is defined by
\begin{align}
  \label{eq-xt-def}
  \mathbf{x}^T =\frac{1}{\sqrt{n}}(n_1 - b_1 n, n_2 - b_2 n, n_3 - b_3 n)\,.
\end{align}
\end{enumerate}
\end{theorem}
By CLT it is clear that w.h.p.\ the total number of edges in $G \sim \cG(n,p)$ is $(1+o(1))\lambda n/2$. Furthermore, by results of \cite{Bollobas84} and \cite{Luczak90} (see also \cite{JLR}), w.h.p.\ our graph $G$ has a unique giant component of size $(1 + o(1))(1 - \mu/\lambda)n $.
Altogether, we deduce that the event $B$ happens w.h.p.; assume therefore that $B$ indeed occurs.
Recalling~\eqref{eq-xt-def} where $(x_1,x_2,x_3)$ is given as a function of $(n_1,n_2,n_3)$, define the event $Q$ by
\begin{align*}\mathcal{Q}_{M} &\deq \left\{(n_1, n_2, n_3)\in \N^3:  |x_i|\leq M~\mbox{for $i=1,2,3$} \right\}\,,\\
Q &\deq \big\{(|H|, |\GC|- |H|, |E(H)| - |H|) \in \mathcal{Q}_M\big\}\,.
\end{align*}
By part~\eqref{item-PW-central-limit} of Theorem~\ref{thm-PW-local-limit}, for any fixed $\delta > 0$ there exists some $M > 0$ such that $\P(Q^c) < \delta$ for a sufficiently large $n$. Next, define
\begin{align*}
P_{\max} & = \max_{(n_1, n_2, n_3)\in \mathcal{Q}_M} \P \left(|H| = n_1,\; |\GC| - |H| = n_2,\; |E(H)| - |H| = n_3 \right)\,,\\
P_{\min} & = \min_{(n_1, n_2, n_3)\in \mathcal{Q}_M} \P \left(|H| = n_1,\; |\GC| - |H| = n_2,\; |E(H)| - |H| = n_3 \right)\,.
\end{align*}
It follows from part~\eqref{item-PW-local-limit} of Theorem~\ref{thm-PW-local-limit} that there exists some $c=c(M) > 0$ such that
\begin{equation}\label{eq-P-M-P-m}P_{\max} \leq c \cdot P_{\min}\,,\end{equation}
when $n$ is sufficiently large.
Notice that by definition of $\mathbf{x}$,
\begin{equation*}\#\{n_2 \in \N: |x_2|\leq M \} \geq M \sqrt{n}\,.\end{equation*}
Combined with \eqref{eq-P-M-P-m}, it follows that for any $(n_1, n_2, n_3)\in \mathcal{Q}_M$ we have
\begin{equation}\label{eq-G-local}
\P\left( |\GC|= n_1 + n_2~,~ Q  \given |H| = n_1\right) \leq \frac{c}{M \sqrt{n}}\,.
\end{equation}
With this estimate for $\P(|\GC|=N \given H)$, the numerator in the right-hand-side of \eqref{eq-prob-ratio}, it remains to estimate the denominator, $\P(|\hGC|=N \given H)$.

Recall that, given $H$, the quantity $|\hGC|$ is a sum of $|H|$ i.i.d.\ Borel($\mu$) random variables (each such
variable is the size of a PGW($\mu$)-tree). We now wish to derive a local limit theorem
for $|\hGC|$, to which end we again apply Theorem~\ref{thm-local}:
It is well known (and easy to show, e.g.~\cite{DLP}*{Claim 4.2}) that a Borel($\mu$) variable has expectation $1/(1-\mu)$ and variance $\mu/(1-\mu)^3$ for any $0 < \mu < 1$.
Recalling the definition of $\mathcal{Q}_m$, we are interested in the following range for $n_1$ and $n_2$:
\begin{align*} n_1 &= (1-\mu)(1-\tfrac{\mu}{\lambda})n + c_1 \sqrt{n}&(|c_1| \leq M)\,,\\
n_2 &= \mu(1-\tfrac{\mu}{\lambda})n + c_2 \sqrt{n}&(|c_2| \leq M)\,.
\end{align*}
Applying the local CLT now implies that for some $\delta'>0$,
\begin{equation}\label{eq-tilde-G-local} \P(|\hGC|= n_1+n_2 \given |H|=n_1) \geq \delta'/\sqrt{n}\,.
\end{equation}
Combining \eqref{eq-G-local} and \eqref{eq-tilde-G-local}, we obtain that when $n$ is sufficiently large,
 \[\frac{\P \left(|\GC| = N~,~Q \given |H|\right)}{\P \left(|\hGC| =  N \given |H|\right)} \leq \frac{c}{M \delta'}\,.\]
By \eqref{eq-prob-ratio} (and recalling that conditioned on $|T_i|$ the tree $T_i$ is uniformly distributed among all unlabeled trees of this size, and similarly for $\tilde{T}_i$), we conclude that for some $c' =c'(M)> 0$ and any unlabeled graph $A$
\begin{align}\label{eq-GC-hat-C}
\P (\GC =A\,,Q\,,B \mid H) \leq c'\,\P (\hGC = A \mid H)\,.
\end{align}

We are now ready to conclude the proof of the main theorem. Let $\tGC$ be defined as in Theorem~\ref{mainthm-struct}. For any set of simple graphs $\mathcal{A}$, define
\begin{equation}
  \label{eq-cH-def}
  \mathcal{H} = \left\{H : \P(\GC \in \mathcal{A}\, , Q\,, B\mid \TC = H) \geq (\P(\tGC \in \mathcal{A}))^{1/2}\right\}\,.
\end{equation}

Recall that by definition, $\tGC$ is produced by first constructing its 2-core (first two steps of the description), then attaching to each of its vertices independent PGW($\mu$)-trees. Hence, for any $H$, the graphs $\hGC$ and $\tGC$ have the same conditional distribution given $\TC[\hGC] = \TC[\tGC] = H$. It then follows from \eqref{eq-GC-hat-C},\eqref{eq-cH-def} that for some constant $c''>0$ and any $H \in \mathcal{H}$,
\[\P(\tGC \in \mathcal{A} \mid \TC[\tGC] = H) \geq c'' (\P(\tGC \in \mathcal{A}))^{1/2}\,.\]
Since we have
\[ \P(\tGC \in \mathcal{A}) \geq c'' (\P(\tGC \in \mathcal{A}))^{1/2} \P(\TC[\tGC] \in \mathcal{H}) \,,\]
the assumption that $\P(\tGC \in \mathcal{A}) \to 0$ gives that $\P(\TC[\tGC] \in \mathcal{H}) \to 0$.

At this point, we combine all the contiguity results thus far to claim that, for any family of simple graphs $\mathcal{F}$,
\[ \P(\TC[\tGC] \in \mathcal{F}) = o(1) \mbox{ implies that }\P(\TC \in \mathcal{F})=o(1)\,.\]
Indeed, by definition, the 2-core of $\tGC$ is precisely the Poisson-geometric model, conditioned on the sum of the degrees ($\sum_u D_u \one_{D_u \geq 3}$) being even. Thus, as $\mathcal{F}$ has only simple graphs,
clearly we may consider the model conditioned on producing a simple graph and in particular that $\sum_u D_u \one_{D_u \geq 3}$ is even.
Applying Theorem~\ref{thm-Poisson-contiguity} (contiguity with Poisson-configuration), Theorem~\ref{thm-Lambda-contiguity} (contiguity with Poisson-cloning) and Theorem~\ref{thm-poisson-ER} (contiguity with Erd\H{o}s-R\'enyi graphs), in that order, now gives the above statement.

This fact and the arguments above now give that $\P(\TC \in \mathcal{H}) \to 0$. By the definition of $\mathcal{H}$, we now conclude that
\[\P(\GC \in \mathcal{A}) \leq \P(B^c) + \P(Q^c) + \P(\TC \in \mathcal{H}) + (\P(\tGC \in \mathcal{A}))^{1/2}\,,\]
where the last term converges to $0$ by assumption.
Taking a limit, we get that $\lim\sup_{n \to \infty}\P(\GC \in \mathcal{A}) \leq \delta$ and the proof is completed by letting $\delta \to 0$.
\end{proof}

\end{document}